\documentclass[a4paper,11pt]{elsarticle}
\usepackage{amsmath,amssymb,amsfonts,amsthm,mathtools}
\usepackage[colorlinks=true]{hyperref}
\hypersetup{urlcolor=blue, citecolor=red, linkcolor=blue}
\usepackage{bbm}
\usepackage{graphicx,color}
\usepackage{enumerate}
\usepackage{appendix}

\newtheorem{lemma}{Lemma}
\newtheorem{prop}{Proposition}
\newtheorem{theorem}{Theorem}

\newtheorem{remark}[theorem]{Remark}
\newtheorem{definition}{Definition}
\numberwithin{equation}{section}
\journal{}
\begin{document}
\nocite{*}
\begin{frontmatter}
\title{Solution space characterisation of perturbed linear functional and integrodifferential Volterra convolution equations: Cesàro limits  }

\author{John A. D. Appleby$^a$}
\ead{john.appeby@dcu.ie}
\cortext[cor1]{Corresponding author}

\author{Emmet Lawless$^a$\corref{cor1}}
\ead{emmet.lawless6@mail.dcu.ie}

\affiliation{School of Mathematical Sciences, Dublin City University.}
\address{DCU Glasnevin Campus, Collins Avenue, Dublin 9, Ireland.}
\journal{arxiv}

\begin{abstract}
In this article we discuss the requirements needed in order to characterise the solution space of perturbed linear integro-differential Volterra convolution equations. We highlight in general how the pointwise behaviour of perturbation functions does not necessarily propagate through to the solution which the classical literature seems to suggest. To illustrate this general idea we show the Cesàro mean of the solution can converge even in cases when the Cesàro mean of the perturbation function diverges. Furthermore we provide a characterisation of when such convergence takes place and explicitly identify the limit in terms of the problem data. Additionally we prove how all results can also be applied to perturbed linear functional differential equations.
\end{abstract}
\end{frontmatter}

\section{Introduction}
Analysis of the asymptotic behaviour of both deterministic and stochastic Volterra type integro-differential equations has attracted a lot of attention in recent decades. For deterministic linear equations of non-convolution type, recent works such as \cite{Boykov2024} investigate stability of the zero solution of systems of perturbed Volterra integro-differential equations while for stochastic equations asymptotic mean square stability is addressed in \cite{AP:2021}. Consider the scalar equation,
\begin{equation} \label{eq. x dynamics}
    \dot{x}(t) = \int_{[0,t]}\nu(ds)x(t-s)+f(t), \quad t \geq 0, \\
\end{equation}
where $\nu$ is a finite signed Borel measure and $x(0)=\xi \in \mathbb{R}$ is a given initial condition. Such convolution equations are in general well understood, to wit, many monographs both classic and modern (such as \cite{BrunnerVIE,cor:1973non,Cor90b,GLS}) have been produced providing an in dept study to which the reader is referred for more details. The underlying resolvent is crucial in studying \eqref{eq. x dynamics}, it's dynamics are given by,
\begin{equation} \label{eq. r dynamcis}
    \dot{r}(t)=\int_{[0,t]}\nu(ds)r(t-s), \quad t > 0; \quad r(0)=1.
\end{equation}

As shown in \cite[Theorem 3.9]{GLS} it is necessary to assume $r \in L^1(\mathbb{R}_+;\mathbb{R})$ in order to obtain strong admissibility results, thus in this article we also employ such an assumption\footnote{Recall this is equivalent to $z-\hat{\nu}(z) \neq 0$ for $\text{Re}(z)\geq0$ where $z \in \mathbb{C}$ and $\hat{\nu}$ denotes the Laplace transform of the measure $\nu$ \cite[Theorem 3.3.5]{GLS}.}. If we consider the unperturbed version of \eqref{eq. x dynamics}, i.e with $f=0$, the solution is simply $x(t)=r(t)\xi$, hence the assumption that $r \in L^1(\mathbb{R}_+;\mathbb{R})$ endows the solution with many desirable properties. For example solutions will vanish at infinity, be bounded, $p-$integrable ($p\geq 1$), absolutely continuous etc., a larger list of properties the solution will possess is collected in Gripenberg, Londen and Staffons \cite[Theorem 3.9]{GLS}. It is therefore very natural to ask, can we characterise when such properties are \emph{preserved} after a perturbation is introduced, or in other words can we give necessary and sufficient conditions on the forcing function $f$ to ensure the solution of equation \eqref{eq. x dynamics} is an element of a particular function space $V$. Surprisingly the classical theory does not provide an answer to this question, but a solution was proposed by the authors in \cite{AL:2023(AppliedMathLetters)}. For a given function space $V$ the key for obtaining such a result is a decomposition lemma\footnote{We also require the operation of convolution with $L^1$ functions to map $V$ into itself and for $r \in V$, but these conditions are known from the classical theory.} which states (for some locally integrable function $f$) if $\int_\cdot^{\cdot+\theta}f(s)ds \in V$ for each $\theta \in (0,1]$ then $f=f_1+f_2$ where $f_1 \in V$ and $\int_0^\cdot f_2(s)ds \in V$. Such a result was proven for $V=BC_0(\mathbb{R}_+;\mathbb{R})$ (space of bounded continuous functions that vanish at infinity) in \cite[Lemma 15.9.2]{GLS} and by the authors for $V=L^p(\mathbb{R}_+;\mathbb{R})$ for $p\geq1$ in \cite[Lemma 2]{AL:2023(AppliedMathLetters)}. In this article this takes the form of Lemma \ref{lem. decomposition lemma part 2} which deals with the case that $V=\textup{Ces}(\mathbb{R}_+;\mathbb{R})$ (see Definition \ref{def. space of fucntion with Cesàro limits}).



It is the authors opinion that the results provided in \cite{AL:2023(AppliedMathLetters)} (and also in this work) are by no means restricted to the particular spaces listed above and  that the framework developed can be extended to a wide variety of function spaces including, but not restricted to, those spaces found in \cite[Theorem 3.9]{GLS}. It is in this regard that this article should be understood as further evidence that the methods developed in \cite{AL:2023(AppliedMathLetters)} can be used to characterise non-standard behaviour in the solution of equation \eqref{eq. x dynamics}. 

Developing a general approach to characterise the behaviour of solutions to \eqref{eq. x dynamics} is of interest as such results are fundamental when considering more complicated stochastic extensions such as
\begin{equation} \label{eq. Stochastic X}
    dX(t)=\left(f(t)+\int_{[0,t]}X(t-s)\nu(ds)\right)dt+\sigma(t)dB(t),
\end{equation}
where $B$ is a one dimensional Brownian motion and $\sigma \in C(\mathbb{R}_+;\mathbb{R})$. For background material on stochastic differential equations see \cite{GikSkor:1972}. Equations such as \eqref{eq. Stochastic X}, as well as their finite memory counterpart\footnote{When $\nu$ is supported on a compact interval $[0,\tau]$ for some $\tau >0$.} have been subjected to intense study \cite{AP:2002(ECP),AP:2004(SubExpItoVol),AF:2003(EJP),ApRie:2006(SAA),ChanWilliams:1989}.  Using the framework developed in \cite{AL:2023(AppliedMathLetters)}, it is possible to give sharp characterisations of the behaviour of sample paths in an \textit{almost sure} sense. This is demonstrated by the authors in \cite{AL:2024(SVE_Lp)} in which results from \cite{AL:2023(AppliedMathLetters)} are leveraged to characterise when solutions of \eqref{eq. Stochastic X} are elements of $L^p(\mathbb{R}_+;\mathbb{R})$ for $p\geq 1$ (pathwise) \textit{almost surely} and when the sample paths vanish at infinity \textit{almost surely}. It is expected that results from this article can also be extended in the same manner. The main use of the deterministic results lies in considering the process $Z(t)=X(t)-Y(t)$ where $Y(t)$ is the solution to an SDE, namely: $dY(t)=(f(t)-Y(t))dt+\sigma(t)dB(t)$. With this observation one can show that $Z$ is continuously differentiable and for each fixed path satisfies an equation of type \eqref{eq. x dynamics} \cite[Lemma 3]{AL:2024(SVE_Lp)}. Thus the study of equation \eqref{eq. Stochastic X} reduces to the study of a Markovian SDE and the deterministic equation \eqref{eq. x dynamics}.

The key point outlined in \cite{AL:2023(AppliedMathLetters)} and throughout this article is that if one wishes to obtain such characterisation results for solutions to \eqref{eq. x dynamics}, it is not the pointwise behaviour of the function $f$ that matters but rather it's integral over compact intervals:
\begin{equation} \label{eq. interval average}
   \int_t^{t+\theta}f(s)ds \quad \text{ for }\theta \in (0,1].
\end{equation}
To the best of our knowledge consideration of this quantity in the context of the asymptotic behaviour of perturbed dynamical systems, first appeared in papers of Strauss and Yorke \cite{SY67a,SY67b}. In many situations it is the case that the forcing function may not be well behaved in a pointwise sense but the smoothing effect of integration drastically improves the situation. For instance in \cite{AL:2023(AppliedNumMath)}, the authors provide an example of a continuous function, in which $\limsup_{t \to \infty}f(t)=+\infty$, and yet,
\[
 \lim_{t \to \infty}\int_t^{t+\theta}f(s)ds=0 \text{ for each } \theta \in (0,1].
\]
It is shown in \cite{AL:2023(AppliedMathLetters)} that given such a perturbation function $f$, solutions of \eqref{eq. x dynamics} obey $x(t) \to 0$ as $ t \to \infty$, despite the fact that the perturbation function has a diverging $\limsup$. The theory for Volterra integro-differential equations is intricately linked to the theory of Volterra integral equations, however we make the observation that this phenomena cannot occur with integral equations. By integral equation we refer to the following:
\begin{equation} \label{eq. x dynamics integral eq}
    x(t) = \int_0^tk(t-s)x(s)ds+f(t), \quad t \geq 0. \\
\end{equation}
We highlight that the classical theory for \eqref{eq. x integral eq} is in fact sharp and in this case it is indeed the pointwise behaviour of the perturbation term that governs rather then the interval average \eqref{eq. interval average}. Thus if one considered the integral equation \eqref{eq. x dynamics integral eq} with $k \in L^1(\mathbb{R}_+;\mathbb{R})$ and a perturbation function $f$ as described above, the solution of \eqref{eq. x dynamics integral eq} would have a diverging $\limsup$ while the solution of the integro-differential equation (with the same kernel $k$) would converge to zero. We also remark that this ``\textit{robustness}'' of integro-differential equations is also a phenomena found in so called functional differential equations wherein the system has only a finite memory.

We illustrate this phenomena once again by selecting a non standard, yet interesting type of convergence. We are concerned with when solutions of \eqref{eq. x dynamics} and \eqref{eq. x dynamics integral eq} admit a Cesàro type limit. Assuming $r \in L^1(\mathbb{R}_+,\mathbb{R})$, we show
\begin{equation} \label{eq. cesaro limit of interval average}
        \frac{1}{t}\int_0^t \int_{s}^{s+\theta}f(u)du\,ds  \longrightarrow C \theta \text{ as } t \to \infty \quad \text{ for each }\theta \in (0,1],
\end{equation}
is equivalent to
\begin{equation} \label{eq. cesaro limit of x}
    \frac{1}{t}\int_0^t x(s)ds \longrightarrow \frac{-C}{\nu(\mathbb{R}_+)} \in \mathbb{R} \,\text{ as }\, t \to \infty,
\end{equation}
where $C \in \mathbb{R}$ and $x$ is the solution of \eqref{eq. x dynamics}. Additionally we show that if you alter condition \eqref{eq. cesaro limit of interval average} by imposing that $f$ itself must posses a Cesàro limit then this imposes necessary restrictions on the Cesàro limit of the derivative of $x$. We show,
\begin{equation} \label{eq. cesaro limit of perturbation function}
        \frac{1}{t}\int_0^t f(s)ds \longrightarrow C \in \mathbb{R} \text{ as } t \to \infty
\end{equation}
is equivalent to
\begin{equation}
    \lim_{t \to \infty} \frac{1}{t}\int_0^t x(s)ds = \frac{-C}{\nu(\mathbb{R}_+)}; \quad \lim_{t\to \infty} \frac{1}{t}\int_0^t \dot{x}(s)ds = 0.
\end{equation}
In the latter we note the condition on the Cesàro limit of the derivative cannot be dropped. This highlights the usual situation that conditions imposed directly on the forcing function $f$ are in general \textit{sufficient} to obtain information about the solution, but not \textit{necessary}. However this is once again not the case for solutions to the integral equation \eqref{eq. x dynamics integral eq}, in fact condition \eqref{eq. cesaro limit of perturbation function} is  equivalent to solutions of \eqref{eq. x dynamics integral eq} having a Cesàro limit which in this case is equal to $C(1-1/\nu(\mathbb{R}_+))$. To the best of our knowledge there has been no analysis of the time average of solutions to \eqref{eq. x dynamics} however we mention analogous equations in discrete time have received some attention \cite{AP:2016(EJQTDE),AP:2017}. Despite this apparent lack of literature there has been much work in the context of admissibility theory for so called Marcinkiewicz spaces which are intimately linked to Cesàro limits of functions. For a brief modern introduction to such spaces we recommend the book by Corduneanu \cite[pp.41-48]{Cor:2009} or for a more comprehensive study \cite{Bertrandias:1966}, while for classical admissibility results regarding integral equations we mention \cite{Benes:1965(JMPh)}.

\section{Mathematical Preliminaries}
Let $p\in [1,\infty)$ and $f:\mathbb{R}_+\to\mathbb{R}$. We say $f \in L^p(\mathbb{R}_+;\mathbb{R})$ if $\int_{\mathbb{R}_+}|f(s)|^p ds<+\infty$ and $f \in L^p_{loc}(\mathbb{R}_+;\mathbb{R})$ if $\int_{K}|f(s)|^pds <+\infty$, where $K\subset \mathbb{R}_+$ is compact and both integrals are understood in the Lebesgue sense. If $M(\mathbb{R}_+;\mathbb{R})$ is the space of finite signed Borel measures on $\mathbb{R}_+$, and $\nu\in M(\mathbb{R}_+;\mathbb{R})$, we consider the halfline Volterra equation given by
\begin{equation} \label{eq. x}
    \dot{x}(t)  = \int_{[0,t]}x(t-s) \nu(ds)+f(t), \quad t \geq 0; \quad    x(0)  = \xi \in \mathbb{R}.
\end{equation}
We say $x$ is a solution of \eqref{eq. x} on an interval $(0,T]$ whenever $x$ is locally absolutely continuous, satisfies the initial condition and obeys the dynamics  in \eqref{eq. x} for almost all $t \in (0,T]$. We stipulate throughout, and \emph{without further reference}, that $f \in L^1_{loc}(\mathbb{R}_+;\mathbb{R})$, which guarantees a solution $x$ exists (see Gripenberg et al.\cite[Theorem  3.3.3]{GLS}). In particular, the solution satisfies a variation of constants formula
\begin{equation} \label{eq. VOC x}
    x(t)=r(t)\xi+\int_0^tr(t-s)f(s)ds, \quad t \geq 0,
\end{equation}
where $r$ is the so--called differential resolvent of $\nu$, which is the unique absolutely continuous function from $\mathbb{R}_+$ to $\mathbb{R}$, satisfying
\begin{equation} \label{eq. r}
  \dot{r}(t)=\int_{[0,t]}\nu(ds)r(t-s), \quad t>0; \quad r(0)=1. 
\end{equation}
If $\mu$ is a finite measure on $[0,\infty)$, and $f:[0,\infty)\to\mathbb{R}$ then their convolution is 
\[
(f\ast \mu)(t):=\int_{[0,t]} f(t-s)\mu(ds), \quad t\geq 0.
\]
The convolution of two functions is defined analogously. Given a measure $\mu \in M(\mathbb{R}_+;\mathbb{R})$, we denote $|\mu|$ as the set function which takes each Borel set $E \subseteq \mathbb{R}_+$ and assigns the total variation of $\mu$ on $E$. Recall the total variation of a measure $\mu$ on a set $E$ is given by $|\mu|(E)\coloneqq\sup_{\pi}\sum_{j=1}^n|\mu(E_j)|$
where $\pi$ denotes the collection of finite partitions of the set $E$. For more details on these definitions see \cite[Section 3.5]{GLS}. With a slight abuse of terminology we shall call $|\mu|$ the total variation measure of $\mu$. The following standard estimate \cite[Theorem 3.5.6]{GLS} will be heavily utilised:
\[
\left| (f\ast \mu)(t)\right| \leq \int_{[0,t]} |f(t-s)||\mu|(ds).
\]
The reader should keep in mind that all results presented for solutions of equation \eqref{eq. x} can be emulated for the finite memory problem, which is a type of functional differential equation. In this instance we would fix a constant $\tau >0$ and consider
\begin{align} \label{eq. Functional x }
\dot{x}(t)  =\int_{[-\tau,0]}x(t+u)\mu(du)+f(t), \quad t\geq 0;\quad x(t) = \psi(t), \quad t\leq0
\end{align}
where $\mu \in M([-\tau,0];\mathbb{R})$ and $\psi \in C([-\tau,0];\mathbb{R})$. With the following convention one can essentially use identical proofs for results regarding equations \eqref{eq. x} and \eqref{eq. Functional x }. For any subset of the real line, write $-E:=\{x\in\mathbb{R}:-x\in E\}$. If $\mu \in M([-\tau,0];\mathbb{R})$, we can construct a $\tilde{\mu}\in M([0,\infty);\mathbb{R})$ by writing 
\begin{align} \label{eq. tildemu}
&\tilde{\mu}(E)=\mu(-E), \quad  \text{for any Borel set $E\subseteq [0,\tau]$} \nonumber, \\ &\tilde{\mu}(E)=0, \quad \text{for any Borel set $E$ with $E\cap[0,\tau] =\emptyset$}. 
\end{align}
Let $g:\mathbb{R}\to\mathbb{R}$ be such that $g(t)=0$ for all $t<0$. With this construction we have
\[
\int_{[-\tau,0]} g(t+s)\mu(ds)=\int_{[0,t]}g(t-s)\tilde{\mu}(ds)  =(g \ast \tilde{\mu} )(t), \quad t\geq 0. 
\]
With this notation the resolvent of equation \eqref{eq. Functional x } obeys,
\begin{equation} \label{eq. functional resolvent}
  \dot{r}_{\tau}(t)=\int_{[0,t]}\tilde{\mu}(ds)r_{\tau}(t-s), \quad t>0; \quad r_{\tau}(0)=1; \quad r_{\tau}(t)=0 \quad t<0. 
\end{equation}
Here we use the notation $r$ and $r_\tau$ to distinguish between the resolvent of the Volterra equation \eqref{eq. x} and the functional equation \eqref{eq. Functional x } respectively.\\

If we define $F(t) \coloneqq \int_{[-\tau,0]}\left(\int_s^0r_{\tau}(t+s-u)\psi(u)du\right)\mu(ds)$, the solution of \eqref{eq. Functional x } has representation
\begin{equation} \label{eq. VOC functional x}
    x(t,\psi)=r_{\tau}(t)\psi(0) + F(t)+\int_0^t r_{\tau}(t-s)f(s)ds, 
\end{equation}
for $t \geq0$. The assumption $r_\tau \in L^1(\mathbb{R}_+;\mathbb{R})$ ensures that $F= O(e^{-\alpha t})$ for some $\alpha >0$, this follows from Theorem I.5.4 and Corollary I.5.5 in \cite{Diekmann} paired with Lemma 2.1 in \cite{AMR}. Thus this extra term which does not appear in the solution for the Volterra equation \eqref{eq. VOC x} will not have any contribution towards the Cesàro mean. As we will be primarily concerned with Cesàro limits of functions we provide a precise definition.
\begin{definition} \label{def. space of fucntion with Cesàro limits}
We define the space of Cesàro functions as follows:
\begin{equation} \label{eq. space of fucntion with Cesàro limits}
    \text{Ces}(\mathbb{R}_+;\mathbb{R}) \coloneqq \left\{f \in L^1_{loc}(\mathbb{R}_+;\mathbb{R}) : \lim_{t \to \infty} \frac{1}{t}\int_0^t f(s)ds = C  \text{ for some } C \in \mathbb{R} \right\}.
\end{equation}
Additionally we say $f$ admits a Cesàro limit if $f \in \text{Ces}(\mathbb{R}_+;\mathbb{R})$.
\end{definition}
\begin{remark} \label{rem. L1 functions have a trvial Cesàro limit}
Clearly all integrable functions admit a trivial Cesàro limit. It is in this regard that the extra term ($F$) appearing in \eqref{eq. VOC functional x} does not provide any additional difficulties. To see this consider $y(t)=x(t,\psi)-F(t)$, then $y$ is the solution of the Volterra integro-differential equation,
\[
\dot{y}(t)  = \int_{[0,t]}y(t-s) \tilde{\mu}(ds)+f(t), \quad t \geq 0; \quad    y(0)  = \psi(0).
\]
Thus we can apply the theory developed for equation \eqref{eq. x} and noting as $F(t)$ is integrable, the Cesàro limits of $y(t)$ and $x(t,\psi)$ agree.
\end{remark}
As mentioned we shall always employ the assumption that the resolvent is integrable. The exact value of the integral of $r$ and $r_\tau$ are needed in many subsequent proofs so we include this as a lemma.
\begin{lemma} \label{lem. cesaro mean of r}
Let $r$ and $r_\tau$ be the solutions to equations \eqref{eq. r} and \eqref{eq. functional resolvent} respectively and assume $r,r_\tau \in L^1(\mathbb{R}_+;\mathbb{R})$. Then,
\begin{align*}
    &\int_0^\infty r(s)ds = \frac{-1}{\nu(\mathbb{R}_+)}; \quad  \int_0^\infty \dot{r}(s)ds = -1; \\
    &\int_0^\infty r_\tau(s)ds = \frac{-1}{\mu([-\tau,0])}; \quad  \int_0^\infty \dot{r}_{\tau}(s)ds = -1.
\end{align*}
\end{lemma}
\begin{proof}[Proof of Lemma \ref{lem. cesaro mean of r}]
 We only provide the proof for $r$ as if one employs the convention outlined by \eqref{eq. tildemu} then the result for $r_\tau$ follows immediately. We note the assumption $r \in L^ 
 1(\mathbb{R}_+;\mathbb{R})$ guarantees $r \to 0$ as $t \to \infty$. Thus integrating equation \eqref{eq. r} over the interval $[0,T]$ we see, $r(T)=1+\int_0^T \int_{[0,t]}r(t-s)\nu(ds)dt.$ Sending $T \to \infty$ and applying Fubini's theorem we obtain,
  \begin{align*}
     -1=\int_0^\infty \int_{[0,t]} r(t-s)\nu(ds)dt&= \int_{[0,\infty]}\int_t^\infty r(t-s)dt\,\nu(ds)\\
     &= \int_0^\infty r(s)ds \nu(\mathbb{R}_+).
 \end{align*}
 To show the second assertion we integrate \eqref{eq. r} over $\mathbb{R}_+$ and apply the above calculation, i.e $\int_0^\infty\dot{r}(t)dt=\int_0^\infty \int_{[0,t]} r(t-s)\nu(ds)dt=-1$.
\end{proof}

\section{Cesàro functions}
In this section we discuss some facts about members of the function space $\text{Ces}(\mathbb{R}_+;\mathbb{R})$. The main results of this paper rely on a decomposition of the forcing function $f$. The general result has the following flavour, let $V$ represent some function space of interest.  If $\int_\cdot^{\cdot+\theta} f(s)ds \in V$ for each $\theta \in (0,1]$ then $f=f_1+f_2$ such that $f_1 \in V$ and $\int_0^\cdot f_2(s)ds \in V$. Successfully improving the classical admissibility theory is completely determined by whether or not one can prove such a decomposition result. This is accomplished by Lemmas \ref{lem. decomposition lemma part 1} and \ref{lem. decomposition lemma part 2} below for $V=\text{Ces}(\mathbb{R}_+;\mathbb{R})$.

\begin{lemma} \label{lem. decomposition lemma part 1}
Assume $f \in L^1_{loc}(\mathbb{R_+};\mathbb{R})$. Suppose for all $\theta \in (0,1],$ $t\mapsto \int_t^{t+\theta}f(s)ds \in \textup{Ces}(\mathbb{R}_+;\mathbb{R})$, i.e
    \[
    \lim_{t\to \infty} \frac{1}{t}\int_0^t \int_{s}^{s+\theta}f(u)duds = C(\theta) \in \mathbb{R}.
    \]
    Then for any fixed $\theta$,
    \[
    \sup_{\substack{t \in [1,\infty) \\T \in [0,\theta]}}\left| \frac{1}{t}\int_0^t \int_s^{s+T} f(u)duds-C(T)\right| < \infty.
    \]
\end{lemma}
\begin{proof}[Proof of Lemma \ref{lem. decomposition lemma part 1}]
Introduce for some $m \in \mathbb{N}$, the set,
\[
Q_m \coloneqq \left\{ T \in [0,\theta] \colon \left| \frac{1}{t}\int_0^t \int_s^{s+T} f(u)duds-C(T)\right| \leq 1, \text{ for all } t \geq m\right\}.
\]
First we show this set is measurable. The mapping $t \mapsto \left| \frac{1}{t}\int_0^t \int_s^{s+T} f(u)duds-C(T)\right|$ is continuous thus,
\begin{align*}
    Q_m & = \left\{ T \in [0,\theta] \colon \left| \frac{1}{t}\int_0^t \int_s^{s+T} f(u)duds-C(T)\right| \leq 1, \text{ for all } t \in \mathbb{Q}\cap [m,\infty) \right\}\\
    & = \bigcap_{t \in \mathbb{Q}\cap [m,\infty)} \left\{ T \in [0,\theta] \colon \left| \frac{1}{t}\int_0^t \int_s^{s+T} f(u)duds-C(T)\right| \leq 1 \right\}.
\end{align*}
Now for each $t$ introduce, $F_t(T): [0,\theta] \to \mathbb{R}_+ : T \mapsto \left| \frac{1}{t}\int_0^t \int_s^{s+T} f(u)duds-C(T)\right|$. Then we can write,
\[
Q_m = \bigcap_{t \in \mathbb{Q}\cap [m,\infty)} F_t(T)^{-1}([0,1]).
\]
Thus to show $Q_m$ is measurable we need only show that $F_t(T)$ is a measurable function, but this is immediate as $C(T)$ is the pointwise limit of measurable functions and thus measurable, and the mapping $T \mapsto \frac{1}{t}\int_0^t \int_s^{s+T} f(u)duds$ is continuous, thus $F_t(T)$ is a measurable function and so $Q_m$ is a well defined measurable set. Additionally it follows from our supposition that $\bigcup_{m=1}^\infty Q_m=[0,\theta]$, thus there must exists an $m' \in \mathbb{N}$ such that $Q_{m'}$ has non zero Lebesgue measure. Fix such an $m'$, then by Lemma 15.9.3 in Gripenberg et. al. \cite{GLS}, it must by the case that the set $Q_{m'}-Q_{m'}=\{T-T' : T,T' \in Q_{m'}\}$ contains an open interval $(-\varepsilon,\varepsilon)$. Without loss of generality take $T,T' \in Q_{m'}$ such that $T>T'$, then we have
\begin{align*}
\Big| \frac{1}{t}&\int_0^t \int_s^{s+T-T'} f(u)duds-C(T-T') \Big| \\
& = \left| \frac{1}{t}\int_0^t \int_s^{s+T-T'} f(u)duds-C(T) + C(T') \right|\\
& \leq \left| \frac{1}{t}\int_0^t \int_{s-T'}^{s+T-T'} f(u)duds-C(T) \right| + \left| \frac{1}{t}\int_0^t \int_{s-T'}^{s} f(u)duds-C(T') \right|\\
& \leq 2,
\end{align*}
for all $t \geq m'+\theta$. To see why the first equality holds consider
\[
\frac{1}{t}\int_0^t \int_s^{s+T-T'} f(u)duds = \frac{1}{t}\int_0^t \int_{s-T'}^{s+T-T'} f(u)duds - \frac{1}{t}\int_0^t \int_{s-T'}^{s} f(u)duds. 
\]
Taking limits on both sides yields
\begin{equation} \label{eq. Limit L(theta) is additive}
  C(T-T') = C(T)- C(T').  
\end{equation}
Now as the above bound holds for all $T \in Q_{m'}-Q_{m'}$, we have in particular for all $T \in [0,\varepsilon)$,
\[
\Big| \frac{1}{t}\int_0^t \int_s^{s+T} f(u)duds-C(T) \Big| \leq 2, \quad \text{ for all } t\geq m'+\theta.
\]
Thus we have
\[
\sup_{\substack{t \in [1,\infty) \\T \in [0,\theta]}}\left| \frac{1}{t}\int_0^t \int_s^{s+T} f(u)duds-C(T)\right| < \infty.
\]  
\end{proof}
\begin{remark}
    The identity \eqref{eq. Limit L(theta) is additive} tells us that when the mapping $t \mapsto \int_t^{t+ \theta}f(s)ds \in \text{Ces}(\mathbb{R}_+;\mathbb{R})$, the limit is an additive function of the parameter $\theta$. This along with measurability ensures the limit has the form $C \theta$ for some $C \in \mathbb{R}$, \cite[Theorem 1.1.8]{BG:RegularVariation}. In the sequel we shall always write the limit in this explicit form.
\end{remark}
\begin{lemma} \label{lem. decomposition lemma part 2}
Let $f \in L^1_{loc}(\mathbb{R_+};\mathbb{R})$. For each $\theta \in (0,1]$, the following are equivalent,

\begin{enumerate}
    \item[(i)]
    \[
    \lim_{t\to \infty} \frac{1}{t}\int_0^t \int_{s}^{s+\theta}f(u)duds = C \theta.
    \]
    \item[(ii)] $f=f_1+f_2$, where $f_1 \in C(\mathbb{R}_+;\mathbb{R})$ and $f_2 \in L^1_{loc}(\mathbb{R_+};\mathbb{R})$ such that,
    \begin{align*}
        \lim_{t \to \infty} \frac{1}{t}\int_0^tf_1(u)du= C, && \lim_{t \to \infty} \frac{1}{t}\int_0^t \int_0^sf_2(u)duds= \frac{C  \theta}{2}.
    \end{align*} 
\end{enumerate}
\end{lemma}
\begin{proof}[Proof of Lemma \ref{lem. decomposition lemma part 2}]
Fix a $\theta \in (0,1]$, extend $f(t)$ to be zero when $t<0$ and set $f_1(t)=\frac{1}{\theta}\int_{t-\theta}^t f(u)du$. Thus by hypothesis we immediately obtain the first limit while the local integrability of $f$ ensures $f_1$ is continuous on the positive real half line. Next we set $f_2=f-f_1$ which yields,
\begin{equation} \label{eq. identity for integral of f_2}
    \int_0^tf_2(s)ds = \int_0^tf(s)ds-\int_0^t\frac{1}{\theta}\int_{s-\theta}^sf(u)duds=\frac{1}{\theta}\int_0^\theta\int_{t-v}^tf(s)dsdv.
\end{equation}
The derivation of this identity is detailed in the appendix of \cite{AL:2023(AppliedMathLetters)}. Now consider,
\[
 \frac{1}{t}\int_0^t \int_0^sf_2(u)duds = \frac{1}{\theta}\int_0^\theta  \frac{1}{t}\int_0^t \int_{s-v}^sf(u)dudsdv.
\]
Subtracting $\frac{1}{\theta}\int_0^{\theta} v C \,dv$ from both sides  and taking absolute values yields,
\[
\left| \frac{1}{t}\int_0^t \int_0^sf_2(u)duds - \frac{C}{\theta}\int_0^{\theta} v dv \right| \leq \frac{1}{\theta}\int_0^\theta \left| \frac{1}{t}\int_0^t \int_{s-v}^sf(u)duds-v C \right| dv.
\]
By Lemma \ref{lem. decomposition lemma part 1}, the integrand on the RHS is uniformly bounded, thus we can invoke Arzelà's dominated convergence theorem . Hence,
\[
\lim_{t\to \infty} \left| \frac{1}{t}\int_0^t \int_0^sf_2(u)duds - \frac{1}{\theta}\int_0^{\theta} v C\, dv \right| = 0,
\]
as required and for the converse consider,
\begin{align*}
    \frac{1}{t}\int_0^t\int_{s-\theta}^sf(u)duds = \frac{1}{t}\int_0^t&\int_{s-\theta}^sf_1 (u)duds + \frac{1}{t}\int_0^t\int_{0}^sf_2(u)duds\\
    & - \frac{1}{t}\int_0^t\int_{0}^{s-\theta}f_2(u)duds.
\end{align*}
The two terms involving $f_2$ will cancel when we pass to the limit, so we need only focus on the $f_1$ term. Consider for some $k >0$ and $t > k$,
\begin{align*}
    \frac{1}{t}\int_0^t\int_{s-\theta}^sf_1 (u)duds & = \frac{1}{t}\int_0^t\int_{0}^{s-\theta}f_1(u)duds - \frac{1}{t}\int_0^t\int_{0}^sf_1(u)duds\\
    & = \frac{1}{t}\int_0^k\int_{0}^sf_1(u)duds + \frac{1}{t}\int_k^t\int_{0}^sf_1(u)duds\\
    & \qquad - \frac{1}{t}\int_0^k\int_{0}^{s-\theta}f_1(u)duds -\frac{1}{t}\int_k^t\int_{0}^{s-\theta}f_1(u)duds.
\end{align*}
The first and third terms will vanish as $t \to \infty$ so we focus on the second and fourth.
\begin{align*}
    \frac{1}{t}\int_k^t\int_{0}^sf_1(u)duds & - \frac{1}{t}\int_k^t\int_{0}^{s-\theta}f_1(u)duds \\
    & = \frac{1}{t}\int_k^t s \left[\frac{\int_{0}^sf_1(u)du}{s}-C\right]ds + \frac{C}{t} \int_k^t s \,ds\\
    & \quad - \frac{1}{t}\int_k^t (s-\theta) \left[\frac{\int_{0}^{s-\theta}f_1(u)du}{s-\theta}-C\right]ds - \frac{C}{t} \int_k^t (s-\theta) \,ds.
\end{align*}
Combining the second and fourth integrals yields,
\begin{align*}
\theta  C  \frac{(t-k)}{t} \quad \longrightarrow \quad \theta  C \quad  \text{ as } t \to \infty.
\end{align*}
Thus to complete the proof we show the terms involving $f_1$ vanish. We have
\begin{align*}
    \frac{1}{t}\int_k^t s \left[\frac{\int_{0}^sf_1(u)du}{s}-C\right]ds- \frac{1}{t}\int_k^t (s-\theta) \left[\frac{\int_{0}^{s-\theta}f_1(u)du}{s-\theta}-C\right]ds.
\end{align*}
Letting $u=s-\theta$ in the second integral and combining the two yields,
\[
\frac{1}{t}\int_{t-\theta}^t s \left[\frac{\int_{0}^sf_1(u)du}{s}-C\right]ds-\frac{1}{t}\int_{k-\theta}^k s \left[\frac{\int_{0}^sf_1(u)du}{s}-C\right]ds.
\]
The second term will vanish and the first term can be estimated by,
\[
\left|\frac{1}{t}\int_{t-\theta}^t s \left[\frac{\int_{0}^sf_1(u)du}{s}-C\right]ds\right| \leq \int_{t-\theta}^t  \left|\frac{\int_{0}^sf_1(u)du}{s}-C\right|ds.
\]
The integrand on the right hand side converging to zero forces the entire integral to vanish and thus the theorem is proven.
\end{proof}
With this fundamental decomposition at hand we state and prove one more crucial lemma before providing our main result.
\begin{lemma} \label{lem. Convolution with finite measure or L1 function}
    Let $\nu \in M(\mathbb{R}_+)$ and $g \in L^1(\mathbb{R}_+)$.  Suppose $f \in \textup{Ces}(\mathbb{R}_+;\mathbb{R})$ with Cesàro limit, $C \in \mathbb{R}$. Then,
    \[
    \lim_{t \to \infty} \frac{1}{t}\int_0^t (f \ast \nu )(s)ds = C \nu(\mathbb{R}_+); \quad  \lim_{t \to \infty} \frac{1}{t}\int_0^t (f \ast g )(s)ds = C \int_0^\infty g(s)ds.
    \]
\end{lemma}
\begin{proof}[Proof of Lemma \ref{lem. Convolution with finite measure or L1 function}]
We only prove the statement for convolutions with finite measures, as the result for $L^1$ functions is a special case wherein the measure is absolutely continuous with respect to the Lebesgue measure. We show the difference $\left|\frac{1}{t}\int_0^t (f \ast \nu )(s)ds - C \nu(\mathbb{R}_+)\right|$, tends to zero as $t \to \infty$. Let $F(t)\coloneqq \int_0^t f(s)ds$, then this difference (after an application of Fubini's theorem) becomes,
\[
\left| \frac{1}{t}\int_{[0,t]} F(t-u) \nu(du)-C \nu(\mathbb{R}_+)\right|.
\]
We first consider only the case when $|C|>0$ as the case when $C=0$ is simpler and follows from an analogous argument. Let $\varepsilon \in (0,1)$ be arbitrary and fix the positive constants $ K, T$ and $\alpha$ such that,
\begin{align*}
    \left|\frac{1}{K}\int_0^Kf(s)ds-C \right| < \frac{\varepsilon}{5 |\nu|(\mathbb{R}_+)}; &&
    |\nu|([T,\infty))  < \frac{\varepsilon}{5|C|};
\end{align*}
\begin{align*}
    \alpha  < \frac{\varepsilon}{5}  \min\left\{\frac{1}{|C| |\nu|(\mathbb{R}_+)}, \left(1-\frac{K}{K+1}\right)\right\}.
\end{align*}
Then fix,
\[
t > \max\left\{\left(K+1\right),\left(T+K\right),\frac{T}{\alpha},\left(\frac{5 |\nu|(\mathbb{R}_+)}{\varepsilon} \sup_{s \in [0,K]}\left|\int_0^s f(s)ds\right|\right)\right\}.
\]
Now consider,
\begin{align*}
    \Big| \frac{1}{t}\int_{[0,t]}  & F(t-u) \nu(du)-C \nu(\mathbb{R}_+)\Big|\\
    & \leq \left| \frac{1}{t}\int_{[0,t-K]} F(t-u) \nu(du)-C \nu(\mathbb{R}_+)\right| + \frac{|\nu|(\mathbb{R}_+)}{t} \sup_{s \in [0,K]}\left|\int_0^s f(s)ds\right|\\
    & < \left| \frac{1}{t}\int_{[0,t-K]} F(t-u) \nu(du)-C \nu(\mathbb{R}_+)\right| + \frac{\varepsilon}{5}.
\end{align*}
Next we make the substitution $s=t-u$,
\begin{align*}
   \Big| \frac{1}{t}\int_{[0,t-K]} & F(t-u) \nu(du)-C \nu(\mathbb{R}_+)\Big| + \frac{\varepsilon}{5}\\
   & = \Big| \frac{1}{t}\int_{[K,t]} s \left[\frac{F(s)}{s}-C\right] \nu(ds) + \frac{C}{t} \int_{[K,t]}s \nu(ds) -C\nu(\mathbb{R}_+) \Big| + \frac{\varepsilon}{5}\\
   & \leq \Big|\frac{C}{t} \int_{[K,t]}s \nu(ds) -C\nu(\mathbb{R}_+) \Big| + \int_{[K,t]}\left|\frac{F(s)}{s}-C\right||\nu|(ds) + \frac{\varepsilon}{5}\\
   & < \Big|\frac{C}{t} \int_{[K,t]}s \nu(ds) -C\nu(\mathbb{R}_+) \Big| + \frac{\varepsilon |\nu|([K,t])}{5 |\nu|(\mathbb{R}_+)} + \frac{\varepsilon}{5}\\
   & \leq \Big|\frac{C}{t} \int_{[K,t]}s \nu(ds) -C\nu(\mathbb{R}_+) \Big| + \frac{2\varepsilon}{5}.
\end{align*}
Now we undo the previous substitution and set $u=t-s$,
\begin{align*}
    \Big|\frac{C}{t} \int_{[K,t]} & s \nu(ds) -C\nu(\mathbb{R}_+) \Big| + \frac{2\varepsilon}{5}\\
    & = \Big|\frac{C}{t} \int_{[0,t-K]}(t-u) \nu(du) -C\nu(\mathbb{R}_+) \Big| + \frac{2\varepsilon}{5}\\
    & \leq \Big| C \nu([0,t-K])-C \nu(\mathbb{R}_+)\Big| + \Big|\frac{C}{t} \int_{[0,t-K]}u\,\nu(du) \Big| + \frac{2\varepsilon}{5}\\
    & < \frac{\varepsilon}{5} + \Big|\frac{C}{t} \int_{[0,t-K]}u\,\nu(du) \Big| + \frac{2\varepsilon}{5},
\end{align*}
where the last inequality follows from the fact that $t>T+K$ and
\[
| C \nu([0,t-K])-C \nu(\mathbb{R}_+)|=|C| |\nu|([t-K,\infty)).
\]
Now our definition of $\alpha$ ensures $\alpha t < t-K$ when $t> K+1$, thus we can write
\begin{align*}
    \Big|\frac{C}{t} \int_{[0,t-K]} & u\,\nu(du) \Big| + \frac{3\varepsilon}{5} \\
    & = \Big|\frac{C}{t} \int_{[0,\alpha t]}u\,\nu(du) +\frac{C}{t} \int_{[\alpha t,t-K]}u\,\nu(du) \Big| + \frac{3\varepsilon}{5}\\
    & \leq \alpha\,|C| |\nu|(\mathbb{R_+})+ \frac{|C|}{t}\int_{[\alpha t,t-K]} |u|\, |\nu|(du) + \frac{3\varepsilon}{5}\\
    & <\frac{\varepsilon}{5} +\frac{|C|}{t}\int_{[\alpha t,t]} |u|\, |\nu|(du) + \frac{3\varepsilon}{5}\\
    & <\frac{\varepsilon}{5} +|C||\nu|([\alpha t,\infty)) + \frac{3\varepsilon}{5}.
\end{align*}
Now finally as $\alpha t > T$ we have, $|C| |\nu|([\alpha t,\infty)) + \frac{4\varepsilon}{5} < \frac{\varepsilon}{5} + \frac{4\varepsilon}{5} = \varepsilon$.
\end{proof}
\section{Results}
We can now state and prove the main result of the paper.
\begin{theorem} \label{thm. cesaro limit of x}
Let $x$ be the solution of equation \eqref{eq. x} and assume $r \in L^1(\mathbb{R_+};\mathbb{R})$. Then the following statements are true,
    \begin{itemize}
        \item[(i)] $ \int_\cdot^{\cdot+\theta}f(s)ds \in \textup{Ces}(\mathbb{R}_+;\mathbb{R})$ for all $\theta \in (0,1] \iff x \in \textup{Ces}(\mathbb{R}_+;\mathbb{R})$.
        \item[(ii)] \,\,$f \in \textup{Ces}(\mathbb{R}_+;\mathbb{R}) \iff x \in \textup{Ces}(\mathbb{R}_+;\mathbb{R})$ and $\dot{x} \in \textup{Ces}(\mathbb{R}_+;\mathbb{R})$.
    \end{itemize}
    
Moreover in each case the limits are given explicitly, that is as $t \to \infty$,
\begin{itemize}
    \item[(i)]
    \begin{align*}
        \frac{1}{t}\int_0^t\int_{s}^{s+\theta}f(u)duds \longrightarrow C \theta; &&  \frac{1}{t}\int_0^t x(s)ds \longrightarrow \frac{-C}{\nu(\mathbb{R}_+)}.
    \end{align*}
    \item[(ii)]
        \begin{align*}
        \frac{1}{t} \int_0^tf(s)ds \longrightarrow C; &&  \frac{1}{t}\int_0^t x(s)ds \longrightarrow \frac{-C}{\nu(\mathbb{R}_+)}; &&  \frac{1}{t}\int_0^t \dot{x}(s)ds \longrightarrow 0.
    \end{align*}
\end{itemize}
    where $C \in \mathbb{R}$.
\end{theorem}
\begin{proof}[Proof of Theorem \ref{thm. cesaro limit of x}]
First we prove $(i)$. Similarly to the proof of Theorem 1 in \cite{AL:2023(AppliedMathLetters)}, using the variation of constants formula \eqref{eq. VOC x}, Lemma \ref{lem. decomposition lemma part 2} and integrating by parts,
\[
x(t) = r(t)\xi+(r\ast f_1)(t) +f_3(t)+(\dot{r} \ast f_3)(t), \quad t\geq 0,
\]
where $f_3(t):=\int_0^tf_2(s)ds$. Thus we have,
\[
\frac{1}{t}\int_0^t x(s)ds = \frac{1}{t}\int_0^t r(s)\xi ds+\frac{1}{t}\int_0^t(r\ast f_1)(s)ds +\frac{1}{t}\int_0^tf_3(s)ds+\frac{1}{t}\int_0^t(\dot{r} \ast f_3)(s)ds.
\]
Sending $t \to \infty$ and applying Lemmas \ref{lem. decomposition lemma part 2} and \ref{lem. Convolution with finite measure or L1 function} we see,
\[
\lim_{t \to \infty} \frac{1}{t}\int_0^t x(s)ds =  \frac{-C}{\nu(\mathbb{R}_+)}.
\]
For the converse we consider an extended solution so that $x(t)=0$ for all $t<0$, then we integrate \eqref{eq. VOC x} over the interval $[t-\theta,t]$ which results in,
\[
\int_{t-\theta}^tf(u)du=x(t)-x(t-\theta)-\int_{t-\theta}^t \int_{[0,s]}x(s-u)\nu(du)ds.
\]
Introducing the notation $X(t)\coloneqq (x \ast \nu)(t)$ we obtain,
\begin{align*}
   \frac{1}{t}\int_0^t \int_{s-\theta}^s f(u)duds= \frac{1}{t}\int_0^t x(s)ds - \frac{1}{t}\int_0^t x(s-\theta)ds - \frac{1}{t}\int_0^t \int_{s-\theta}^s X(u)duds.
\end{align*}
When passing to the limit the first two terms on the RHS will cancel, thus we need only focus on the third. Now recall by Lemma \ref{lem. Convolution with finite measure or L1 function},
\[
\frac{1}{t}\int_0^t X(s)ds \longrightarrow \frac{-C}{\nu(\mathbb{R}_+)}  \nu(\mathbb{R}_+)=-C,
\]
as $t \to \infty $. Thus we may follow the proof of the converse of Lemma \ref{lem. decomposition lemma part 2} with $X$ in place of $f_1$ to obtain,
\[
\lim_{t \to \infty} \frac{1}{t}\int_0^t \int_{s-\theta}^s X(u)duds =- \theta  C.
\]
But this implies,
\[
\lim_{t \to \infty} \frac{1}{t}\int_0^t \int_{s-\theta}^s f(u)duds = \theta C,
\]
as required. For the forward implication of $(ii)$, recall equation \eqref{eq. VOC x}, $x(t)=r(t)\xi + (r \ast f)(t)$. As $r \in L^1(\mathbb{R}_+;\mathbb{R})$ it's Cesàro limit will be zero, thus,
\[
\lim_{t \to \infty} \frac{1}{t}\int_0^t x(s)ds = \lim_{t \to \infty} \frac{1}{t}\int_0^t (r \ast f)(s)ds= \frac{-C}{\nu(\mathbb{R}_+)},
\]
where the last equality follows from Lemma \ref{lem. cesaro mean of r} and \ref{lem. Convolution with finite measure or L1 function}. We now use this to show the Cesàro limit for $\dot{x}$. By \eqref{eq. x} we obtain,
\[
\frac{1}{t}\int_0^t \dot{x}(s)ds = \frac{1}{t}\int_0^t (x \ast \nu)(s)ds + \frac{1}{t}\int_0^t f(s)ds.
\]
Passing to the limit and applying Lemma \ref{lem. Convolution with finite measure or L1 function} and the limit just proven for $x$ yields,
\[
\frac{1}{t}\int_0^t \dot{x}(s)ds = \frac{-C}{\nu(\mathbb{R}_+)}\nu(\mathbb{R}_+) + C=0.
\]
For the converse, rearrange equation \eqref{eq. x} to obtain $f(t)=\dot{x}(t)-\int_{[0,t]}x(t-s)\nu(ds)$. By supposition the Cesàro limit of $\dot{x}$ is zero and Lemma \ref{lem. Convolution with finite measure or L1 function} tells us the Cesàro limit of the convolution is $\frac{C}{\nu(\mathbb{R}_+)}\nu(\mathbb{R}_+)=C$, as required.
\end{proof}
The assumptions that $r,r_\tau \in L^1(\mathbb{R}_+;\mathbb{R})$ are known to be sharp in order to obtain admissibility results for solutions to equations \eqref{eq. x} and \eqref{eq. Functional x }. We provide a converse result for the finite memory problem which shows that $r_\tau \in L^1(\mathbb{R}_+;\mathbb{R})$ is a necessary condition in order for solutions of $\eqref{eq. Functional x }$ to admit a Cesàro limit. Throughout the proof we make heavy use of the semi-explicit asymptotic expansion of the resolvent, this approach is much more involved in the Volterra case due to the lack of information about the general form of the Volterra resolvent. It is in this regard that we only consider the finite memory equation.
\begin{prop} \label{prop. r integrable converse}
Let $x$ be the solution of \eqref{eq. Functional x }. If for each $\psi \in C([-\tau,0];\mathbb{R})$ and $f \in L^1_{loc}(\mathbb{R}_+;\mathbb{R})$ such that $t \mapsto \int_t^{t+\theta}f(s)ds \in \textup{Ces}(\mathbb{R}_+;\mathbb{R})$ for all $\theta \in (0,1]$, we have $x(\cdot,\psi) \in \textup{Ces}(\mathbb{R}_+;\mathbb{R})$, then $r_\tau \in L^1(\mathbb{R}_+;\mathbb{R})$.
\end{prop}
For the readers convenience, before we provide a proof of Theorem \ref{prop. r integrable converse} we recall some common notation in the field of functional equations. Let,
\begin{equation} \label{eq. x0 functional equation}
    x_0(t,\psi) \coloneqq r_{\tau}(t)\psi(0) + \int_{[-\tau,0]}\left(\int_s^0r_{\tau}(t+s-u)\psi(u)du\right)\mu(ds),
\end{equation}
so we can rewrite \eqref{eq. VOC functional x} as $ x(t,\psi)=x_0(t,\psi)+(r_\tau \ast f)(t)$. Thus $x_0$ solves equation \eqref{eq. Functional x } with the perturbation term switched off. The roots of the transcendental equation, $h(\lambda)=\lambda -  \int_{[-\tau,0]} e^{-\lambda s}\mu(ds)$, govern the asymptotic behaviour of $r_\tau$. Let $\Lambda = \{ \lambda \in \mathbb{C} : h(\lambda)=0 \}$ and define $v_0(\mu) \coloneqq \sup \{ \text{Re}(\lambda): \lambda \in \Lambda\}$. The cardinality of $\Lambda' = \{ \lambda \in \Lambda : \text{Re}(\lambda) = v_0(\mu) \}$ is finite, thus $r_\tau$ admits the expansion,
\begin{equation} \label{eq. expansion of functional r}
e^{-v_0(\mu) t}r_\tau(t)= \sum_{\lambda_j \in \Lambda'} \left\{ p_j(t)\cos(\beta_j t)+ q_j(t)\sin(\beta_j t)\right\}+g(t),
\end{equation}
where $p_j$ and $q_j$ are polynomials of degree $m_j-1$ (where $m_j$ is the multiplicity of $\lambda_j$), $\beta_j = \text{Im}(\lambda_j)$ and $g(t)=o(e^{-\varepsilon t })$ for some $\varepsilon >0$. For proofs of all facts recalled above the reader may consult \cite[Chapter I]{Diekmann}.
\begin{proof}[Proof of Theorem \ref{prop. r integrable converse}]
    We aim to prove $v_0(\mu)<0$ which yields automatically $r_\tau \in L^1(\mathbb{R}_+;\mathbb{R})$ due to the expansion \eqref{eq. expansion of functional r}. We proceed by contradiction, thus assume $v_0(\mu) > 0$. Setting $f=0$ yields $x_0(\cdot,\psi) \in \textup{Ces}(\mathbb{R}_+;\mathbb{R})$ for all $\psi \in C([-\tau,0];\mathbb{R})$. Hence select $\lambda \in \Lambda'$ and set $\psi= \text{Re}(e^{\lambda t})$ which yields the solution, $x_0(t;\psi)=\text{Re}(e^{\lambda t})$ for all $t \in [-\tau,\infty)$. Let $\lambda=v_0(\mu)+i\beta$, then
    \[
    \int_0^t\text{Re}(e^{\lambda s})ds=\int_0^t e^{v_0(\mu)s}\cos(\beta s)ds=\frac{e^{v_0(\mu)t}\left(\beta\sin(v_0(\mu)t)+v_0(\mu)\cos(\beta t)\right)}{|\lambda|^2},
    \]
    which upon dividing by $t$ will not tend to a limit, contradicting $x_0(\cdot,\psi) \in \textup{Ces}(\mathbb{R}_+;\mathbb{R})$. Thus we must have $v_0(\mu) \leq 0$. Now assume $v_0(\mu)=0$, in which we distinguish two cases. First assume there exists a $\lambda \in \Lambda'$ with multiplicity strictly great than one. In this case we fix $\psi=t\cos(\text{Im}(\lambda)t)$ which again yields a solution $x_0(t;\psi)=t\cos(\text{Im}(\lambda)t)$ for all $t \in [-\tau,\infty)$, but as before a simple calculation shows
    \[
    \frac{1}{t}\int_0^t s\cos(\text{Im}(\lambda)s)ds=\frac{\sin(\text{Im}(\lambda)t)}{\text{Im}(\lambda)}+ \frac{\cos(\text{Im}(\lambda)t)-1}{t\text{Im}(\lambda)^2},
    \]
    which oscillates indefinitely and so $x_0(\cdot;\psi) \notin \text{Ces}(\mathbb{R}_+;\mathbb{R})$, a contradiction. Thus it must be when $v_0(\mu)=0$, all elements of $\Lambda'$ have multiplicity one. In this case,
    \[
    r_\tau(t)= \sum_{j=1}^n \left\{ c_j\cos(\beta_j t)+ k_j\sin(\beta_j t)\right\}+g(t),
    \]
    where $c_j,k_j \in \mathbb{R}\setminus\{0\}$. Now we fix the initial condition $\psi=0$ which yields, $x(t,0)=(r_\tau \ast f)(t)$. Now assume first that $n>1$ which means at least one zero is not at the origin. Without loss of generality we may assume $\beta_1 \neq 0$. Choose $f(t) = k_1\sin(\beta_1 t)-c_1\cos(\beta_1 t)$ and define $F(t)\coloneqq \int_0^tf(s)ds$, we see $\int_0^t x(u,0)du=\int_0^t r_\tau(u)F(t-u)du$ with,
    \[
    F(t)=\frac{c_1}{\beta_1}\cos(\beta_1 t)+\frac{k_1}{\beta_1}\sin(\beta_1 t)-\frac{c_1}{\beta_1}.
    \]
    Now we can calculate $\int_0^t x(u,0)du$ semi-explicitly,
    \begin{align*}
        &\frac{1}{t}\int_0^t r_\tau(u)F(t-u)du \\
        &= \sin(\beta_1 t)\left(\frac{c_1k_1}{\beta_1}+\frac{(c_1^2-k_1^2)}{2t\beta_1^2}\right) + \cos(\beta_1t)\left(\frac{c_1k_1}{t\beta_1^2}+\frac{(c_1^2-k_1^2)}{2\beta_1}\right)-\frac{c_1k_1}{t\beta_1^2}\\
        & \quad +\frac{1}{t} \sum_{j=1}^n \left\{A_{1,j}\cos(\beta_j t)+A_{2,j}\sin(\beta_j t)+A_{3,j}\cos(\beta_1 t)+A_{4,j}\sin(\beta_1 t) + A_{5,j}\right\}\\
        & \quad + \frac{1}{t}\int_0^t g(u)F(t-u)du,
    \end{align*}
     where $A_{i,j}\in \mathbb{R}$. The summation term will vanish as $t \to \infty$ and the exponential estimate on $g$ ensures the integral term will also vanish. However we see the first term will oscillate indefinitely as the constants $c_1,k_1$ and $\beta_1$ are non-zero. This means $x(\cdot;0) \notin \text{Ces}(\mathbb{R}_+;\mathbb{R})$ which is a contradiction. Now if $n=1$ there is only a single root on the imaginary axis. If this root is not at the origin then the above argument still holds, so we assume the sole root is at the origin. In this case $r_\tau(t)=c+g(t)$ where $c \in \mathbb{R}\backslash\{0\}$. Now set $\psi=0$ and $f(t)=1/c$, then $x(t,0)=t+(g \ast f)(t)$. As $g \in L^1(\mathbb{R}_+;\mathbb{R})$ and $f \in \text{Ces}(\mathbb{R}_+;\mathbb{R})$, Lemma \ref{lem. Convolution with finite measure or L1 function} implies the convolution has a Ces\'aro limit but the linear function $f(t)=t$ will not, which yields the final contradiction. Thus we must have that $v_0(\mu)<0$ and the claim is proven.
\end{proof}
Proposition \ref{prop. r integrable converse}, Theorem \ref{thm. cesaro limit of x} and remark \ref{rem. L1 functions have a trvial Cesàro limit}, provide the proof of our next theorem.
\begin{theorem} \label{thm. Main theorem for functional equations}
    Let $x(\cdot,\psi)$ be the solution to \eqref{eq. Functional x }.
    \begin{itemize}
        \item[(\textbf{A})] \,\,Suppose $r_\tau \in L^1(\mathbb{R}_+;\mathbb{R})$, then for all $\psi \in C([-\tau,0];\mathbb{R})$ the following hold true,
        \begin{itemize}
            \item[(i)] $ \int_\cdot^{\cdot+\theta}f(s)ds \in \textup{Ces}(\mathbb{R}_+;\mathbb{R})$ for all $\theta \in (0,1] \iff x(\cdot,\psi) \in \textup{Ces}(\mathbb{R}_+;\mathbb{R})$.
            \item[(ii)]\, $f \in \textup{Ces}(\mathbb{R}_+;\mathbb{R}) \iff x(\cdot,\psi) \in \textup{Ces}(\mathbb{R}_+;\mathbb{R})$ and $\dot{x}(\cdot,\psi) \in \textup{Ces}(\mathbb{R}_+;\mathbb{R})$.
        \end{itemize}
        Moreover in each case the limits are given explicitly, that is as $t \to \infty$,
\begin{itemize}
    \item[(i)]
    \begin{align*}
        \frac{1}{t}\int_0^t\int_{s}^{s+\theta}f(u)duds \to C \theta; &&  \frac{1}{t}\int_0^t x(s,\psi)ds \to \frac{-C}{\mu([-\tau,0])}.
    \end{align*}
    \item[(ii)]
        \begin{align*}
        \frac{1}{t} \int_0^tf(s)ds \to C; &&  \frac{1}{t}\int_0^t x(s,\psi)ds \to \frac{-C}{\mu([-\tau,0])}; &&  \frac{1}{t}\int_0^t \dot{x}(s,\psi)ds \to 0.
    \end{align*}
\end{itemize}
    where $C \in \mathbb{R}$.
        \item[(\textbf{B})]\,\, The following are equivalent,
        \begin{itemize}
            \item[(i)] $r_\tau \in L^1(\mathbb{R}_+;\mathbb{R})$.
            \item[(ii)]\, For all  $\psi \in C([-\tau,0];\mathbb{R})$ and $f \in \textup{Ces}(\mathbb{R}_+;\mathbb{R})$  we have $x(\cdot,\psi) \in \textup{Ces}(\mathbb{R}_+;\mathbb{R})$.
            \item[(iii)]\, For all  $\psi \in C([-\tau,0];\mathbb{R})$ and $f$ s.t. $ \int_\cdot^{\cdot+\theta}f(s)ds \in \textup{Ces}(\mathbb{R}_+;\mathbb{R})$ for all $\theta \in (0,1]$ we have $x(\cdot,\psi) \in \textup{Ces}(\mathbb{R}_+;\mathbb{R})$.
        \end{itemize}
    \end{itemize}
\end{theorem}
\section{Example} \label{sec. example}
In this section we show the statements $f \in \text{Ces}(\mathbb{R}_+;\mathbb{R})$ and $t \mapsto \int_t^{t+\theta}f(s)ds \in \text{Ces}(\mathbb{R}_+;\mathbb{R})$ for each $\theta \in (0,1]$ are in fact not equivalent. In general the first implies the second, to see this recall that if $f \in \text{Ces}(\mathbb{R}_+;\mathbb{R})$ then by Theorem \ref{thm. cesaro limit of x} part $(ii)$, we necessarily have solutions of \eqref{eq. x} obeying $x \in \text{Ces}(\mathbb{R}_+;\mathbb{R})$. But then applying part $(i)$ of Theorem \ref{thm. cesaro limit of x} yields $t \mapsto \int_t^{t+\theta}f(s)ds \in \text{Ces}(\mathbb{R}_+;\mathbb{R})$ for each $\theta \in (0,1]$. This argument, although mathematically sound, is unsatisfactory due to its indirect nature. Indeed this can be proven in a direct fashion with an apt application of Lemma \ref{lem. Convolution with finite measure or L1 function}. Let $f \in \text{Ces}(\mathbb{R}_+;\mathbb{R})$ and introduce the function $g(t)=\theta^{-1}\chi_{\{t \in [0,\theta]\}}(t)$. Then $g \in L^1(\mathbb{R}_+;\mathbb{R})$, hence by Lemma \ref{lem. Convolution with finite measure or L1 function} $(f\ast g) \in \text{Ces}(\mathbb{R}_+;\mathbb{R})$, but for $t>\theta$,
\[
(f\ast g)(t)=\int_0^\theta g(s)f(t-s)ds= \frac{1}{\theta}\int_{t-\theta}^t f(s)ds.
\]
Thus we have proven as claimed that,
\begin{equation} \label{eq. forward implication of Cesaro space inclusion}
    f \in \text{Ces}(\mathbb{R}_+;\mathbb{R}) \implies t \mapsto \int_t^{t+\theta}f(s)ds \in \text{Ces}(\mathbb{R}_+;\mathbb{R}) \text{ for all } \theta \in (0,1].
\end{equation}
We state a counter example that shows the reverse implication is in general false, consider $f:\mathbb{R}_+\to \mathbb{R}: t \mapsto t^{\alpha+1}\sin(t^{\alpha+1}) \quad \alpha >0$. This claim can be verified after some tedious but elementary calculations. Although the reverse implication of \eqref{eq. forward implication of Cesaro space inclusion} is not true in general we provide a side condition which yields the converse in the case that the function is positive.
\begin{prop} \label{prop. equivalence of Cesàro limits for f>0}
Let $f \in L^1_{loc}(\mathbb{R}_+;\mathbb{R}_+)$. The following are equivalent,
\begin{align*}
     & (i) \lim_{t \to \infty}\frac{1}{t}\int_0^t\int_{s}^{s+\theta}f(u)duds = C \theta  \text{ for all } \theta \in (0,1]; \quad \lim_{t \to \infty}\frac{1}{t}\int_t^{t+1}f(s)ds=0.\\
    & (ii)     \lim_{t \to \infty} \frac{1}{t}\int_0^t f(s)ds = C.
\end{align*}
\end{prop}
\begin{proof}[Proof of proposition \ref{prop. equivalence of Cesàro limits for f>0}]
The fact that $(ii)$ implies $(i)$ is clear from the discussion at the beginning of this section and from the identity,
\[
\frac{1}{t}\int_t^{t+1} f(s)ds = \frac{1}{t}\int_0^{t+1} f(s)ds-\frac{1}{t}\int_0^t f(s)ds.
\]
To see the other implication we write for $t> \theta$,
\begin{align}
    \frac{1}{t}\int_0^t\int_{s}^{s+\theta}& f(u)duds\\ \nonumber
    &= \frac{1}{t}\int_0^\theta f(u)(u-\theta)du+\frac{\theta}{t}\int_0^tf(u)du+\frac{1}{t}\int_t^{t+\theta}f(u)(t+\theta-u)du.
\end{align}
The first and third term on the right hand side will vanish upon sending $t \to \infty$ and thus we obtain the desired limit.
\end{proof}
\section{Comparison with integral equations}
To this point we have considered integro-differential and functional equations and have shown solutions to perturbed equations may admit a Cesàro limit even when the perturbation function does not. This phenomena is impossible for integral equations, in this section we provide a proof. Consider the integral equation,
\begin{equation} \label{eq. x integral eq}
    x(t) = \int_0^tk(t-s)x(s)ds+f(t), \quad t \geq 0, \\
\end{equation}
where $f \in L^1_{loc}(\mathbb{R}_+;\mathbb{R})$ and $k \in L^1(\mathbb{R}_+;\mathbb{R})$. In this regime the integral resolvent of $k$ denoted by $r_k$ is the solution to the convolution equation $r_k=k+r_k\ast k$, on the halfline $\mathbb{R}_+$. The standard variation of constants formula yields, $x(t)=f(t)+(r_k \ast f)(t), t \geq 0$.
Now let $r_k \in L^1(\mathbb{R}_+;\mathbb{R})$. Assume $f \in \text{Ces}(\mathbb{R}_+;\mathbb{R})$ then Lemma \ref{lem. Convolution with finite measure or L1 function} implies $x \in \text{Ces}(\mathbb{R}_+;\mathbb{R})$. Conversely assume $x \in \text{Ces}(\mathbb{R}_+;\mathbb{R})$, then with $k \in L^1(\mathbb{R}_+;\mathbb{R})$ we can rearrange \eqref{eq. x integral eq} and apply Lemma \ref{lem. Convolution with finite measure or L1 function} to give $f \in \text{Ces}(\mathbb{R}_+;\mathbb{R})$.
\section{Conclusion} \label{sec. conclusion}
In this paper we have characterised when the solution of a perturbed linear functional differential equation lies in the space $\text{Ces}(\mathbb{R}_+;\mathbb{R})$ and also provided necessary and sufficient conditions to ensure the solution of a perturbed linear integro-differential Volterra equation lies in $\text{Ces}(\mathbb{R}_+;\mathbb{R})$\footnote{under the assumption of an integrable resolvent}. It is interesting to reflect on the improvements made with respect to the classical theory. For classical admissibility theory, three things are needed to ensure solutions of \eqref{eq. x} reside in a particular space V while operating under the assumptions that $r\in L^1(\mathbb{R}_+;\mathbb{R})$ and $\nu \in M(\mathbb{R}_+;\mathbb{R})$, namely:
\begin{itemize}
    \item[1.] The resolvent obeys, $r \in V$.
    \item [2.] If $g \in L^1(\mathbb{R}_+;\mathbb{R})$ then the operator $h \mapsto g \ast h$ must map $V$ into itself.
    \item[3.] The perturbation function obeys $f \in V$.
\end{itemize}
This paper along with \cite{AL:2023(AppliedMathLetters)} have shown that the third point can be significantly weakened. We can have $f \notin V$ and yet solutions still reside in $V$. This is contingent on the perturbation function admitting a decomposition of the form $f=f_1+f_2$ such that $f_1 \in V$ and $t\mapsto \int_0^tf_2(s)ds \in V$. Even after scrutiny of the proofs of Lemmas \ref{lem. decomposition lemma part 1} and \ref{lem. decomposition lemma part 2} it seems unclear how one may generally describe such spaces. However even with this said, it is the authors opinion that given any ``\textit{reasonable}'' space that is usually considered for admissibility theory that satisfies point 2 above, a decomposition result like the one described can be proven. Following the authors recent works \cite{AL:2023(AppliedNumMath),AL:2023(AppliedMathLetters),AL:2024(SVE_Lp)}  there is mounting evidence to suggest that in order to obtain sharp results for perturbed dynamical systems (deterministic or stochastic) the correct condition to study is \eqref{eq. interval average}. An interesting extension would be to consider emulating such methodology for infinite dimensional equations like those found in Prüss \cite{Pruss}.\\

\textbf{Acknowledgements:} EL is supported by Science Foundation Ireland (16/IA/4443). JA is supported by the RSE Saltire Facilitation Network on Stochastic Differential Equations: Theory, Numerics and Applications (RSE1832). EL would like to thank Ole Cañadas for insightful discussions resulting in the improvement of this work.

\end{document}